\documentclass{amsart}
\usepackage{url}
\usepackage{graphicx}
\newtheorem{thm}{Theorem}[section]
\newtheorem{cor}[thm]{Corollary}
\newtheorem{lem}[thm]{Lemma}
\newtheorem{prop}[thm]{Proposition}
\theoremstyle{definition}
\newtheorem{defn}[thm]{Definition}
\theoremstyle{remark}
\newtheorem{rem}[thm]{Remark}
\newtheorem{ex}[thm]{Example}

\begin{document}
\title[Almost product structures on statistical manifolds]{Almost product structures on statistical manifolds and para-K\"{a}hler-like statistical submersions}
\author[G.E. V\^{\i}lcu]{Gabriel Eduard V\^{\i}lcu}

\date{}

\abstract  The main purpose of the present work is to investigate statistical manifolds endowed with almost product structures. We prove that the statistical structure of a para-K\"{a}hler-like statistical manifold of constant curvature in the Kurose's sense is a Hessian structure. We also derive the main properties of statistical submersions which are compatible with almost product structures.
The results are illustrated by several nontrivial examples. 
\endabstract

\maketitle

\section{Introduction}

A statistical manifold is a semi-Riemannian manifold $(M,g)$ equipped with an additional structure given by a pair of torsion-free affine connections $(\nabla,\nabla^*)$ which are dual with respect to $g$. This concept provides a setting for the field of information geometry, a domain having deep links with several research areas \cite{ABKL,BOY,MR}.

On the other hand, the notion of Riemannian submersion, which is the dual of the notion of isometric immersion, was introduced by O'Neill in \cite{ON1} and Gray \cite{GR}. Later, Riemannian submersions between manifolds endowed with various geometric structures were studied by many authors (see, e.g., \cite{FIP,SAH} and the references therein). Statistical submersions between statistical manifolds have been introduced and investigated in \cite{AH}. Statistical manifolds equipped with remarkable geometric structures, as well as statistical submersions between such manifolds were also studied \cite{AQS,FHOSS,MS,NOD,TAK2,TAK4}. In the present paper, we investigate statistical manifolds equipped with almost product structures, introducing the notions of para-K\"{a}hler-like statistical manifold and para-K\"{a}hler-like statistical submersion.  Recall that para-K\"{a}hler structures were first investigated by Rashevski \cite{RAS}, under the name of  stratified spaces. However, the explicit definition and the main properties of para-K\"{a}hler manifolds  were given in \cite{ROZ,RUS}. We note that these manifolds were also studied by Libermann in the context of
$G$-structures \cite{LIB}. Presently, there is a great interest in this subject (see, e.g., the recent papers \cite{AG,IC,Petr}), due to the fact that para-K\"{a}hler structures are related to a number of interesting topics, not only from mathematics (see, for instance, the very interesting survey \cite{CFG}), but also from mechanics and theoretical physics \cite{COR,KT}. Moreover, very recently, Fei and Zhang \cite{FJ} defined the concept of almost Codazzi-para-K\"{a}hler manifold, investigating the interaction of Codazzi couplings (see \cite{TZ}) with para-K\"{a}hler geometry. In particular, they proved that any statistical manifold admits a para-K\"{a}hler structure in some very natural conditions and discussed the significance of the result under the context of information geometry and theoretical physics, giving a strong motivation for studying such kind of statistical manifolds (see Section 4 in \cite{FJ}).

The present work is organized as follows. Section 2 contains definitions and
basic properties of statistical manifolds and statistical submersions. In section 3 we investigate statistical manifolds with almost product structures and introduce the concept of para-K\"{a}hler-like statistical manifold. In one of the main results, we prove that the statistical structure of a para-K\"{a}hler-like statistical manifold of constant curvature in the Kurose's sense is a Hessian structure.
Section 4 is devoted to the study of the para-K\"{a}hler-like statistical submersions. We study the transference of the structures defined on the total space of the submersion and investigate the geometry of the fibers and base space. Several illustrative examples are given in the last section of the article. Apart from examples constructed explicitly in coordinates, we use the Sasaki metric to provide para-K\"{a}hler-like statistical structures on the tangent bundles of para-Hermitian-like manifolds, and we show that the statistical manifolds corresponding to some standard statistical models, like the well-known normal distribution, Poisson distribution, multinomial distribution, multivariate normal distribution, Dirichlet distribution and Von Mises-Fisher distribution, can be equipped with such structures.

\section{Preliminaries}

In this section we provide the most important definitions and notations for
our framework based mainly on \cite{AH,AMN,ON3}.

Let $(M,g)$ be a semi-Riemannian manifold and $\nabla$ be an affine connection on $M$. The \emph{conjugate affine connection} $\nabla^*$ of $\nabla$ with respect to the metric $g$ is given by
\begin{equation}\label{1}
Eg(F,G)=g(\nabla_EF,G)+g(F,\nabla^*_EG),
\end{equation}
for all $E,F,G\in\Gamma(TM)$, where $\Gamma(TM)$ denotes the set of smooth tangent vector fields on $M$.

The pair $(\nabla,g)$  is said to be a \emph{statistical structure} on $M$ if the torsion tensor field of $\nabla$ vanishes and the covariant derivative $\nabla g$ is symmetric. In this case, the triple $(M,\nabla,g)$ is said to be a \emph{statistical manifold}. It is easy to see that if $(M,\nabla,g)$ is a statistical manifold, so is  $(M,\nabla^*,g)$ \cite{FH}. Moreover,  the triple $(M,\nabla^*,g)$ is said to be the \emph{dual statistical manifold} of $(M,\nabla,g)$ and the triple $(\nabla,\nabla^*,g)$ is called the \emph{dualistic structure} on $M$.

We also recall that the affine connections $\nabla$ and $\nabla^*$ are called \emph{dual connections} \cite{VOS} due to the fact that $(\nabla^*)^*=\nabla$.
On the other hand, it is easy to see that the dual connections $\nabla$ and $\nabla^*$ are related by \cite{ZHANG}
\begin{equation}\label{2}
\nabla+\nabla^*=2\nabla^{0},
\end{equation}
where $\nabla^{0}$ is Levi-Civita connection of the metric $g$. Hence, obviously, the geometry of statistical manifolds simply reduces to the semi-Riemannian geometry when $\nabla$ and $\nabla^*$ coincide \cite{NOG}. In this case, the pair $(\nabla^{0},g)$ is said to be a \emph{Riemannian statistical
structure} or a \emph{trivial statistical structure}. 

For the affine connection $\nabla$, we set the curvature tensor field $R^\nabla$ with sign convention
\[
R^\nabla(E,F)G=\nabla_E\nabla_FG-\nabla_F\nabla_EG-\nabla_{[E,F]}G
\]
for all vector fields $E,F,G$ on $M$. For the sake of simplicity, we denote shortly $R^\nabla$ by $R$ and similarly we denote the curvature tensor field $R^{\nabla^*}$ of the dual connection $\nabla^*$ by $R^*$. A statistical structure $(\nabla,g)$ is called a \emph{Hessian structure}  if $\nabla$ is flat, that is, the curvature tensor field $R$ identically vanishes. For a deeper study of the Hessian geometry, the reader can refer to \cite{BOYOM2,SHIMA}.

We remark that $R$ does not have the property of skew-symmetry relative to $g$, i.e.
\[
g(R(E,F)G,G')\neq-g(R(E,F)G',G).
\]
Hence, it is inappropriate to define the sectional curvature of a statistical manifold using $R$. However,
according to \cite{FH,OP}, for a statistical manifold $(M,\nabla,g)$, we can define the \emph{statistical curvature tensor field} $S$  by
\[
S(E,F)G=\frac{1}{2}[R(E,F)G+R^*(E,F)G]
\]
for $E,F,G\in\Gamma(TM)$. It is easy to see now that $S$
is a Riemann-curvature-like tensor \cite{OP2}: is skew-symmetric in $E,F$, satisfies the first Bianchi identity and is skew-symmetric relative to $g$. 

\begin{defn}\cite{FH}
Let $(M,\nabla,g)$ be a statistical manifold. For $p\in M$ and a non-degenerate 2-dimensional subspace $\Pi={\rm span}_{\mathbb{R}}\{v,w\}$ of $T_pM$, the expression
\[
k(p,\Pi)=\frac{g(S_p(v,w)w,v)}{g(v,v)g(w,w)-g(v,w)^2}
\]
is called the \emph{sectional curvature} of $(M,\nabla,g)$ for $\Pi$. The statistical manifold $(M,\nabla,g)$ is said to be \emph{of constant sectional curvature} $k$, where $k\in\mathbb{R}$, if $k(p,\Pi)$ is constant for all non-degenerate 2-dimensional subspaces $\Pi\subset T_pM$ and for all $p\in M$.
\end{defn}

It is known that the sectional curvature of a statistical manifold $(M,\nabla,g)$ is constant $k$ if and only if the statistical curvature tensor field $S$ satisfies \cite{FH}
\[
S(E,F)G=k\{g(F,G)E-g(E,G)F\}
\]
for all $E,F,G\in\Gamma(TM)$.

On the other hand, a statistical structure $(\nabla,g)$ on a manifold $M$ is said to be \emph{of constant curvature} if the curvature tensor field $R$ with respect to the affine connection $\nabla$ satisfies \cite{KUR}
\begin{equation}\label{scc}
R(E,F)G=k\{g(F,G)E-g(E,G)F\}
\end{equation}
for all $E,F,G\in\Gamma(TM)$. We say in this case that the statistical manifold $(M,\nabla,g)$ is a \emph{space of constant curvature in the Kurose's sense} \cite{FH}.

Due to the fact that $R$ and $R^*$ are related by
\[
g(R(G,G')E,F)=-g(R^*(G,G')F,E)
\]
for all $E,F,G,G'\in\Gamma(TM)$, it follows that if $(M,\nabla,g)$ is a space of constant curvature in the Kurose's sense, so is the dual statistical manifold $(M,\nabla^*,g)$. Hence, we deduce in this case that the sectional curvature of the statistical manifold $(M,\nabla,g)$ is constant $k$.

\begin{rem}\label{remi}
Let $M=\{p_{\xi}|\xi\in\Xi\}$ be an $n$-dimensional statistical model, that is a family of probability distributions $p_\xi=p(x;\xi)$, where $\xi=(\xi^1,\ldots,\xi^n)$ runs through an open domain $\Xi$ in $\mathbb{R}^n$ and $x$ runs through a measure space $\chi$ with measure $dx$ so that
$\int_\chi p(x;\xi)dx=1$ for each $\xi$. Then we may consider $M$ as a $C^{\infty}$ differentiable manifold  and we can define a statistical structure $(\nabla^{(\alpha)},g)$ on $M$, where
$g$ is the Fisher metric of the statistical model defined by
\[
g_{ij}=\mathbb{E}[\partial_il\partial_jl],
\]
where $\mathbb{E}$ is the mean, $\partial_i=\frac{\partial}{\partial_{\xi_i}}$, $l=l(x;\xi)=\log p(x;\xi)$, and $\nabla^{(\alpha)}$ is a connection on $M$, called  $\alpha$-connection, given by
\[
g(\nabla^{(\alpha)}_{\partial_i}\partial_j,\partial_k)=\mathbb{E}\left[\left(\partial_i\partial_jl+\frac{1-\alpha}{2}\partial_il\partial_jl\right)\partial_kl\right],
\]
where $\alpha$ is some arbitrary real number. It is easy to see that the $\alpha$-connection is torsion-free and $\nabla^{(-\alpha)}$ is conjugate of $\nabla^{(\alpha)}$ relative to the Fisher metric. Hence $(M,\nabla^{(\alpha)},g)$ is a statistical manifold of the statistical model $M$. We remark that the $0$-connection is the Levi-Civita connection with respect to the Fisher metric (see \cite{AMN}).
\end{rem}

\begin{defn}\cite{ON3}
Let $(M,g)$ and $(M',g')$ be two connected semi-Riemannian manifolds
of index $s$ and $s'$
respectively, with $0\leq s\leq {\rm dim} M$, $0\leq s'\leq {\rm dim} M'$ and $s'\leq s$. Then a smooth map $\pi:M\rightarrow M'$ which is onto is said to be a \emph{semi-Riemannian
submersion} if the following conditions are satisfied:

(i) $\pi_{*|p}: T_pM \rightarrow T_{\pi(p)}M'$ is onto for all $p\in
M$;

(ii) The fibres $\pi^{-1}(p'),\ p'\in M'$, are semi-Riemannian
submanifolds of $M$;

(iii) $\pi_*$ preserves scalar products of vectors normal to fibres.
\end{defn}

As usual we call the vectors tangent to fibres as \emph{vertical} vectors and those normal
to fibres as \emph{horizontal} vectors. We denote by $\mathcal{V}$ the
vertical distribution, by $\mathcal{H}$ the horizontal distribution
and by $v$ and $h$ the vertical and horizontal projection. An
horizontal vector field $X$ on $M$ is said to be \emph{basic} if $X$ is
$\pi$-related to a vector field $X'$ on $M'$.
It is clear that every
vector field $X'$ on $M'$ has a unique horizontal lift $X$ to $M$
and $X$ is basic. Note that the basic vector fields locally span the horizontal distribution. Moreover, if $X$ and $Y$ are basic vector fields on $M$, $\pi$-related to $X'$ and
$Y'$ on $M'$, then we have the following properties \cite{FIP,ON3,SAH}:

(i) $g(X,Y)=g'(X',Y')\circ\pi$;

(ii) $h[X,Y]$ is a basic vector field and
\[\pi_*h[X,Y]=[X',Y']\circ\pi.\]

(iii) For any vertical vector field $V$, $[X,V]$ is vertical.

Next, let $(M,\nabla,g)$ be a statistical manifold, $(M',g')$ a semi-Riemannian manifold and let $\pi:M\rightarrow M'$ be a semi-Riemannian
submersion. Then we denote by $\widehat{\nabla}$ and $\widehat{\nabla}^*$  the affine connections induced on fibres by the dual connections $\nabla$ and $\nabla^*$ from $M$. We remark that $\widehat{\nabla}$ and $\widehat{\nabla}^*$  may be defined as
\begin{equation}\label{3}
\widehat{\nabla}_UV=v\nabla_UV,\ \widehat{\nabla}^*_UV=v\nabla^*_UV
\end{equation}
for all $U,V\in\Gamma(\mathcal{V})$. Moreover, it follows immediately that the connections $\widehat{\nabla}$ and $\widehat{\nabla}^*$ are torsion free and conjugate to each other with respect to the induced metric on fibres. On the other hand, if we define
\[K:=\nabla-\nabla^*,\]
then $K$ is symmetric, and the following property holds \cite{TAK2}
\begin{eqnarray}\label{4}
2g(\nabla_XY,Z)&=&g(K_XY,Z)+Xg(Y,Z)+Yg(Z,X)-Zg(X,Y)\nonumber\\
&&-g(X,[Y,Z])+g(Y,[Z,X])+g(Z,[X,Y])
\end{eqnarray}
for all $X,Y,Z\in\Gamma(\mathcal{H})$.

Similarly, if $\nabla'$ and $\nabla'^*$ are affine connections on $M'$, then we can define
\[K'=\nabla'-\nabla'^*\]
and we have that $hK_XY$ is basic and $\pi$-related to $K'_{X'}Y'$ if and only if $h\nabla_XY$, respectively  $h\nabla^*_XY$, is basic and $\pi$-related to $\nabla'_{X'}Y'$, respectively $\nabla^*_{X'}Y'$.

\begin{defn} \cite{TAK4}
Let $(M,\nabla,g)$ and $(M',\nabla',g')$ be two statistical manifolds. Then a semi-Riemannian submersion $\pi:M\rightarrow M'$ is said to be a \emph{statistical submersion} if
\[
\pi_{*}(\nabla_XY)_{p}=(\nabla'_{X'}Y')_{\pi(p)}
\]
for all basic vector fields $X,Y$ on $M$ which are $\pi$-related to $X'$ and
$Y'$ on $M'$, and $p\in M$.
\end{defn}

We note that if $\pi:M\rightarrow M'$ is a statistical submersion, then any fiber is a statistical manifold \cite{AH,TAK2,TAK4}.
Moreover, we can define as
well as in the case of classical pseudo-Riemannian geometry, two (1,2) tensor
fields $T$ and $A$ on $M$, by the formulas
\begin{equation}\label{5}
       T(E,F)=T_EF=h\nabla_{vE}vF+v\nabla_{vE}hF
       \end{equation}
and
\begin{equation}\label{6}
       A(E,F)=A_EF=v\nabla_{hE}hF+h\nabla_{hE}vF,
       \end{equation}
for $E,F\in \Gamma(TM)$.

Similarly, we can also define the tensor fields $T^*$ and $A^*$ on $M$ by replacing $\nabla$ by $\nabla^*$ in equations \eqref{5} and \eqref{6}.
Using the above definitions one can easily prove the following result.
\begin{lem}\cite{AH}
The tensor fields $T$, $A$, $T^*$ and $A^*$ have the following properties:
\begin{equation}\label{7}
       T_UV=T_VU,\ T^*_UV=T^*_VU,
       \end{equation}
\begin{equation}\label{8}
        A_XY-A_YX=A^*_XY-A^*_YX=v[X,Y],
       \end{equation}
\begin{equation}\label{9}
       A_XY=-A^*_YX,
       \end{equation}
\begin{equation}\label{10}
        \nabla_XY=h\nabla_XY+A_XY,\ \nabla^*_XY=h\nabla^*_XY+A^*_XY,
       \end{equation}
\begin{equation}\label{11}
        \nabla_UV=T_UV+\widehat{\nabla}_UV,\ \nabla^*_UV=T^*_UV+\widehat{\nabla}^*_UV,
       \end{equation}
\begin{equation}\label{12}
        \nabla_UX=h\nabla_UX+T_UX,\ \nabla^*_UX=h\nabla^*_UX+T^*_UX,
       \end{equation}
\begin{equation}\label{13}
        \nabla_XU=A_XU+v\nabla_XU,\ \nabla^*_XU=A^*_XU+v\nabla^*_XU,
       \end{equation}
\begin{equation}\label{14}
        g(T_UV,X)=-g(V,T^*_UX),
       \end{equation}
\begin{equation}\label{15}
        g(A_XY,U)=-g(Y,A^*_XU),
       \end{equation}
for all $X,Y\in\Gamma(\mathcal{H})$ and $U,V\in\Gamma(\mathcal{V})$.
\end{lem}

We recall that if $T_UV=0$, for all $U,V\in\Gamma(\mathcal{V})$, then $\pi$ is said to be a statistical submersion \emph{with isometric fibers} \cite{TAK2}.

\section{Statistical manifolds with almost product structures}

An \emph{almost product structure} on a smooth manifold $M$ is a tensor
field $P$ of type (1,1) on $M$, $P\neq\pm {\rm Id}$, such that:
\begin{equation}\label{16}
         P^2={\rm Id},
\end{equation}
where ${\rm Id}$ is the identity tensor field of
type $(1,1)$ on $M$. The pair $(M,P)$ is called an \emph{almost product manifold}. An \emph{almost para-complex manifold} is an almost product manifold $(M,P)$ such that the two eigenbundles $T^+M$ and $T^-M$ associated
with the two eigenvalues $+1$ and $-1$ of $P$, respectively, have the
same rank. A nice survey of recent results on para-complex geometry is included in \cite{AMT}.

An \emph{almost para-Hermitian structure} on a smooth manifold $M$
is a pair $(P,g)$, where $P$ is an almost product structure on $M$
and $g$ is a semi-Riemannian metric on $M$ satisfying
\begin{equation}\label{17}
         g(PX,PY)=-g(X,Y),
         \end{equation}
for all vector fields $X$,$Y$ on $M$.

In this case, $(M,P,g)$ is said to be an \emph{almost para-Hermitian
manifold}. It is easy to see that the dimension of $M$ is even and
the metric $g$ is neutral. Moreover, if $\nabla P=0$ then $(M,P,g)$ is said to be a
\emph{para-K\"{a}hler manifold} \cite{CHEN}.

\begin{rem}
 If $(M,g)$ is a semi-Riemannian manifold endowed with an almost product structure $P$, then we denote by $P^*$ 
 the tensor field of type $(1,1)$ on $M$  satisfying
\begin{equation}\label{18}
         g(P E,F)+g(E,P^* F)=0,
\end{equation}
for all vector fields $E,F$ on $M$.
Obviously $P^*$ is the negative of the adjoint of $P$ and hence exists.
Next, we will call the triple $(M,P,g)$ as an \emph{almost para-Hermitian-like manifold}, motivated by the fact that in the particular case when $P^*=P$, the condition \eqref{18} reduces to \eqref{17}.
\end{rem}

\begin{defn}
Let  $(M,P,g)$ be an almost para-Hermitian-like manifold. If $(\nabla,g)$ is a statistical structure on $M$ such that $P$ is parallel with respect to $\nabla$, then $(M,\nabla,P,g)$ is said to be a \emph{para-K\"{a}hler-like statistical manifold}.
\end{defn}

\begin{rem}\label{R1}
The concepts of almost para-Hermitian-like  manifold and para-K\"{a}hler-like statistical manifold generalize the notions of almost para-Hermitian manifold and para-K\"{a}hler manifold, respectively. It is clear that in the particular case of para-K\"{a}hler manifolds we have $P=P^*$ and $\nabla$ is the Levi-Civita connection of the metric $g$.
\end{rem}

\begin{prop}\label{P1}
Let $(M,\nabla,P,g)$ be an almost para-Hermitian-like manifold. Then:
\begin{enumerate}
  \item[i.] $P^*$ is an almost product structure;
  \item[ii.] For all vector fields $E,F$ on $M$, the following formula holds
   \begin{equation}\label{19}
         g(P E,P^*F)=-g(E,F);
\end{equation}
  \item[iii.] $(P^*)^*=P$;
  \item[iv.] If $(\nabla,g)$ is a statistical structure on $M$, then $P$ is parallel with respect to $\nabla$ if and only if $P^*$ is parallel with respect to the conjugate affine connection $\nabla^*$ of $\nabla$.
\end{enumerate}
\end{prop}
\begin{proof}
i. Using \eqref{18} and taking into account that $P$ is an almost product structure, we obtain
\begin{eqnarray}
g((P^*)^2E,F)&=&-g(P^*E,PF)=g(E,P^2F)\nonumber\\
&=&g(E,F),\nonumber
\end{eqnarray}
for all vector fields $E,F$ on $M$. Now, we deduce immediately that $P^*$ is also an almost product structure. \\
ii. The relation \eqref{19} follows immediately from i.\\
iii. Suppose that $(P^*)^*=P'$. Then we have
\begin{equation}\label{21}
         g(P^* E,F)+g(E,P' F)=0.
\end{equation}
From \eqref{18} and \eqref{21} it follows that
\[
g(E,(P'-P)F)=0
\]
for all vector fields $E,F$ on $M$. Therefore we conclude that $P'=P$.\\
iv. Using \eqref{18} and the definition of the conjugate connection $\nabla^*$ of $\nabla$, we obtain
\begin{eqnarray}
g((\nabla_EP)F,G)&=&g(\nabla_EPF,G)-g(P\nabla_EF,G)\nonumber\\
&=&Eg(PF,G)-g(PF,\nabla^*_EG)+g(\nabla_EF,P^*G)\nonumber\\
&=&Eg(PF,G)+g(F,P^*\nabla^*_EG)+Eg(F,P^*G)-g(F,\nabla^*_EP^*G)\nonumber\\
&=&g(F,P^*\nabla^*_EG)-g(F,\nabla^*_EP^*G)\nonumber\\
&=&-g(F,(\nabla^{*}_{E}P^*)G),\nonumber
\end{eqnarray}
for all vector fields $E,F,G$ on $M$.
Hence we derive that $P$ is parallel with respect to $\nabla$ if and only if $P^*$ is parallel with respect to $\nabla^*$.
\end{proof}

From Proposition \ref{P1} we deduce immediately the following result.
\begin{cor}  $(M,\nabla,P,g)$ is a para-K\"{a}hler-like statistical manifold if and only if so is $(M,\nabla^*,P^*,g)$.
\end{cor}


\begin{defn}
Let $(M,\nabla,P,g)$ be a para-K\"{a}hler-like statistical manifold.
If the curvature tensor $R$ of the connection $\nabla$ satisfies
\begin{eqnarray}\label{22}
R(E,F)G&=&\frac{c}{4}\lbrace g(F,G)E-g(E,G)F\nonumber\\
        &+&
        g(PF,G)PE-g(PE,G)PF\nonumber\\
        &+&
        [g(E,PF)-g(PE,F)]PG\rbrace,
\end{eqnarray}
for all vector fields $E,F,G$ on $M$, where $c$ is a real constant, then $(M,\nabla,P,g)$  is said to be a statistical manifold \emph{of type para-K\"{a}hler space form}.
\end{defn}

We remark that replacing $P$ by $P^*$ in the right-hand side of the equation \eqref{22}, we get the expression of the curvature tensor $R^*$ with respect to the dual connection $\nabla^*$. We note that the above definition generalizes the concept of para-K\"{a}hler space form; in fact, if $(M,P,g)$ is a para-K\"{a}hler manifold satisfying \eqref{22}, then it follows immediately that $M$ is a \emph{para-K\"{a}hler manifold of constant para-sectional curvature}, also called \emph{para-K\"{a}hler space form} (see, e.g., \cite{CHEN2011}). Notice that one has
a rich family of para-K\"{a}hler space forms (see \cite{GM}).


\begin{thm}\label{TK}
Let  $(M,\nabla,P,g)$  be a para-K\"{a}hler-like statistical manifold of dimension $n\neq 2$. If $M$ is a space of constant curvature in the Kurose's sense, then the statistical structure $(\nabla,g)$ is a Hessian structure.
\end{thm}
\begin{proof}
Because $M$ is a space of constant curvature in the Kurose's sense, it follows that the curvature tensor field $R$ with
respect to the affine connection $\nabla$ is given by \eqref{scc}. On the other hand, because $P$ is parallel with respect to $\nabla$, it is obvious that
\[
R(E,F)PG=PR(E,F)G,
\]
for all vector fields $E,F,G$ on $M$. Hence we deduce
\begin{equation}\label{pro1}
g(R(E,F)PG,G')=g(PR(E,F)G,G')
\end{equation}
for all vector fields $E,F,G$ on $M$.

Using now \eqref{scc} in \eqref{pro1}, we derive
\begin{eqnarray}\label{pro2}
k\{g(F,PG)g(E,G')-g(E,PG)g(F,G')\}\nonumber\\
=
k\{g(F,G)g(PE,G')-g(E,G)g(PF,G')\}.
\end{eqnarray}

We assume now that $k\neq 0$ in \eqref{scc}. Then we take an orthonormal tangent frame $\{e_1,\ldots,e_n\}$ such that $g(e_i,e_j)=\epsilon_i\delta_{ij}$ on $M$ and replace $F=G=e_i$ in \eqref{pro2}. By summing over $1\leq i\leq n$ and taking account of \eqref{18}, we deduce
\begin{eqnarray}\label{pro3}
({\rm trace\ }P) g(E,G')+g\left(\sum_{i=1}^n\epsilon_i g(P^*E,e_i)e_i,G'\right)\nonumber\\
=
n g(PE,G')+g\left(E,\sum_{i=1}^n\epsilon_i g(e_i,P^*G')e_i\right).\nonumber
\end{eqnarray}

Thus we have
\begin{equation}\label{pro4}
({\rm trace\ }P) g(E,G')+g(P^*E,G')=n g(PE,G')+g(E,P^*G').
\end{equation}

Using now \eqref{18} in \eqref{pro4}, we obtain
\begin{equation}\label{pro5}
({\rm trace\ }P) g(E,G')-g(E,PG')-(n-1) g(PE,G')=0,
\end{equation}
for all vector fields $E,G'$ on $M$.

Replacing in \eqref{pro5} $E$ by $G'$ and $G'$ by $E$, we derive
\begin{equation}\label{pro6}
({\rm trace\ }P) g(E,G')-g(PE,G')-(n-1) g(E,PG')=0.
\end{equation}

By subtracting \eqref{pro5} from \eqref{pro6}, we deduce
\[
(n-2)[g(PE,G')-g(E,PG')]=0
\]
and taking into account that $n\neq 2$, we get
\begin{equation}\label{pro7}
g(PE,G')=g(E,PG').
\end{equation}

From \eqref{pro5} and \eqref{pro7}, we derive
\begin{equation}\label{pro8}
({\rm trace\ }P) g(E,G')-ng(PE,G')=0,
\end{equation}
for all vector fields $E,G'$ on $M$.

Now, it follows immediately from \eqref{pro8} that
\begin{equation}\label{pro9}
PE=\frac{{\rm trace\ }P}{n}E
\end{equation}
and therefore we derive
\begin{equation}\label{pro10}
P^2E=\left(\frac{{\rm trace\ }P}{n}\right)^2E.
\end{equation}

But $P$ is an almost product structure, so we deduce from \eqref{pro10} that
\begin{equation}\label{pro11}
{\rm trace\ }P=\pm n.
\end{equation}

From \eqref{pro9} and \eqref{pro11}, we derive that $P=\pm{\rm Id}$. This is a contradiction because $P$ is an almost product structure and, by definition, $P\neq\pm{\rm Id}$.

Hence the assumption $k\neq 0$ in \eqref{scc} is false and we conclude that
the curvature tensor field $R$ with
respect to the affine connection $\nabla$ identically vanishes. Therefore we deduce that $(\nabla,g)$ is a Hessian structure.
\end{proof}


\section{Semi-Riemannian submersions between para-K\"{a}hler-like statistical manifolds}

\begin{defn} \cite{SCHAF}
Let  $(M,P,g)$ and $(M',P',g')$ be two almost para-Hermitian-like manifolds. Then a smooth map $f:M\rightarrow M'$ is said to be a \emph{para-holomorphic map} if
\begin{equation}\label{23BISBB}
        f_*\circ P=P'\circ f_*.
\end{equation}
\end{defn}

\begin{defn}
Let $(M,P,g)$ and $(M',P',g')$ be two almost para-Hermitian-like manifolds. Suppose that $(\nabla,g)$ and $(\nabla',g')$ are statistical structures on  $M$ and $M'$, respectively. Then:\\
i. A statistical submersion $\pi:M\rightarrow M'$ which is a para-holomorphic map is called an \emph{almost para-Hermitian-like statistical submersion}.\\
ii. If $(M,\nabla,P,g)$ is a para-K\"{a}hler-like statistical manifold, then an almost para-Hermitian-like statistical submersion $\pi:M\rightarrow M'$  is called a \emph{para-K\"{a}hler-like statistical submersion}.
\end{defn}

\begin{prop}\label{P111}
Let $\pi:M\rightarrow M'$ be an almost para-Hermitian-like statistical submersion. Then:

i. $\mathcal{V}$ and $\mathcal{H}$ are
invariant under the action of $P$.

ii. $P$ and $P^*$
commute with the horizontal and vertical projectors.

iii. If $X$ is a basic vector field on $M$ $\pi$-related to $X'$ on $M'$, then $P X$ (resp. $P^* X$) is a basic vector field $\pi$-related to $P' X'$ (resp. $P'^* X'$) on $M'$.
\end{prop}
\begin{proof}
i. Since $\pi$ is a para-holomorphic map, we derive for
any $V\in\Gamma(\mathcal{V})$:
\[
\pi_*P V=P'\pi_*V=0
\]
and thus we conclude that $P(\mathcal{V})\subseteq \mathcal{V}$. Similarly it follows that $P^*(\mathcal{V})\subseteq \mathcal{V}$.

On the other hand, for any $X\in\Gamma(\mathcal{H})$ and
$V\in\Gamma(\mathcal{V})$, we derive
\[
g(P X,V)=-g(X,P^* V)=0
\]
and thus we conclude that $P(\mathcal{H})\subseteq \mathcal{H}$. In a similar way, it follows that $P^*(\mathcal{H})\subseteq \mathcal{H}$.

ii. The statement trivially follows from i.

iii. If $X$ is a basic vector field, then from i. it follows that $P X$ and  $P^*X$ are horizontal vector fields. On the other hand, since  $\pi$ is a para-holomorphic map and $X$ is $\pi$-related to $X'$ on $M'$, we obtain
\[
\pi_*P X=P'\pi_* X=P'X'
\]
and similarly
\[
\pi_*P^* X=P'^*\pi_* X=P'^* X'
\]
and the conclusion is clear.
\end{proof}

\begin{rem}
Let $(M,P,g)$ and $(M',P',g')$ be two almost para-Hermitian-like manifolds. Suppose that $(\nabla,g)$ and $(\nabla',g')$ are statistical structures on  $M$ and $M'$, respectively. Let $\pi:M\rightarrow M'$ be an almost para-Hermitian-like statistical submersion. If $F=\pi^{-1}(p')$ is a fiber of the submersion, where $p'\in M'$, then it is known from \cite{AH,TAK2,TAK4} that $(F,\widehat{\nabla},\widehat{g}=g|_F)$ is a statistical manifold. Moreover, we can define
$
\widehat{P}:=P|_F,
$
and then it follows immediately that $(F,\widehat{P},\widehat{g})$ is an almost para-Hermitian-like manifold. Hence we have the following result.
\end{rem}

\begin{thm}\label{T1}
If $\pi:M\rightarrow M'$ is an almost para-Hermitian-like statistical submersion, then each fiber is an almost para-Hermitian-like manifold endowed with a statistical structure.
\end{thm}

\begin{thm}\label{T2}
Let  $(M,\nabla,P,g)$ be a para-K\"{a}hler-like statistical manifold and let $(M',P',g')$ be an almost para-Hermitian-like manifold endowed with a statistical structure $(\nabla',g')$.
If $\pi:M\rightarrow M'$ is a para-K\"{a}hler-like statistical submersion, then $(M',\nabla',P',g')$ is a para-K\"{a}hler-like statistical manifold. Moreover, the fibres are also para-K\"{a}hler-like statistical manifolds.
\end{thm}
\begin{proof}

If $X,Y$ are basic vector fields on $M$ $\pi$-related to $X',Y'$ on $M'$, then using Proposition \ref{P111} we obtain
    \begin{eqnarray}\label{23}
    (\nabla'_{X'}P')Y'&=&\nabla'_{X'}(P'Y')-P'(\nabla'_{X'}Y')\nonumber\\
    &&=\nabla'_{\pi_*X}(\pi_*(P Y))-P'\pi_*(\nabla_X Y)\nonumber\\
    &&=\pi_*(\nabla_X(P Y))-\pi_*(P(\nabla_X Y))\nonumber\\
    &&=\pi_*((\nabla_X P)Y).
    \end{eqnarray}

Now, because $(M,\nabla,P,g)$ is a para-K\"{a}hler-like statistical manifold, we have
that $P$ is parallel with respect to $\nabla$ and then from \eqref{23} we derive that $P'$ is also parallel with respect to $\nabla'$. Hence $(M',\nabla',P',g')$ is a para-K\"{a}hler-like statistical manifold.

Next, let $F=\pi^{-1}(p')$, $p'\in M'$, be a fiber of the submersion. Then from Theorem \ref{T1} we have that $(F,\widehat{P},\widehat{g})$  is an almost para-Hermitian-like manifold equipped with a statistical structure $(\widehat{\nabla},\widehat{g})$. By using \eqref{11} we deduce
    \begin{equation}\label{26}
    (\nabla_UP)V=(\widehat{\nabla}_UP)V+(T_U\widehat{P} V-P T_UV),\
    \end{equation}
for all $U,V\in\Gamma(\mathcal{V})$.

On the other hand, because $P$ is parallel with respect to $\nabla$, we derive from \eqref{26}
    \begin{equation}\label{28}
    (\widehat{\nabla}_UP)V=0
    \end{equation}
and
    \begin{equation}\label{29}
    T_U\widehat{P} V=P T_UV.
    \end{equation}

Finally, we conclude from \eqref{28} that $(F,\widehat{\nabla},\widehat{P},\widehat{g})$ is a para-K\"{a}hler-like statistical manifold.
\end{proof}

\begin{thm}\label{T3}
Let  $(M,\nabla,P,g)$ be a para-K\"{a}hler-like statistical manifold and let $(M',P',g')$ be an almost para-Hermitian-like manifold endowed with a statistical structure $(\nabla',g')$. If  $\pi:M\rightarrow M'$ is a para-K\"{a}hler-like statistical submersion, then:

i. $T_{\widehat{P}U}\widehat{P}V=T_UV$, for all $U,V\in\Gamma(\mathcal{V})$;

ii. $A_XY=A^*_XY=0$, for all $X,Y\in\Gamma(\mathcal{H})$, provided that ${\rm rank}(\widehat{P}+\widehat{P}^*)$ coincides with the dimension of the fibers.
\end{thm}
\begin{proof}
i. Since $T$ has the symmetry property for vertical vector fields (cf. \eqref{7}), using \eqref{16} and \eqref{29} we derive for all $U,V\in\Gamma(\mathcal{V})$:
\begin{eqnarray}
    T_{\widehat{P} U}\widehat{P} V&=&P T_{\widehat{P} U}V\nonumber\\
    &=&P T_V \widehat{P} U\nonumber\\
    &=&P^2 T_VU\nonumber\\
    &=&T_VU\nonumber\\
    &=&T_UV.\nonumber
    \end{eqnarray}

ii.  Using \eqref{10} we obtain
    \begin{equation}\label{30}
    (\nabla_XP)Y=(h\nabla_XP)Y+(A_XPY-\widehat{P} A_XY).
    \end{equation}
for all $X,Y\in\Gamma(\mathcal{H})$.

Now, because $P$ is parallel with respect to $\nabla$, we deduce from \eqref{30} that
    \begin{equation}\label{31}
    (h\nabla_XP)Y=0
    \end{equation}
and
    \begin{equation}\label{32}
    A_XP Y=\widehat{P} A_XY.
    \end{equation}

    Similarly, we deduce that
    \begin{equation}\label{33}
    A^*_XP^* Y=\widehat{P}^* A^*_XY,
    \end{equation}
for all $X,Y\in\Gamma(\mathcal{H})$.

Next, we take $X,Y\in\Gamma(\mathcal{H})$ and $U\in\Gamma(\mathcal{V})$. Then, making use of \eqref{9}, \eqref{15}, \eqref{18}, \eqref{32}  and \eqref{33}, we obtain
\begin{eqnarray}\label{33b}
    g((\widehat{P}+\widehat{P}^*)A_XY,U)&=&g(\widehat{P}A_XY,U)+g(\widehat{P}^*A_XY,U)\nonumber\\
    &=&g(A_XPY,U)-g(\widehat{P}^*A^*_YX,U)\nonumber\\
    &=&-g(PY,A^*_XU)-g(A^*_Y{P}^*X,U)\nonumber\\
    &=&g(Y,P^*A^*_XU)+g(A_{P^*X}Y,U).
    \end{eqnarray}

On the other hand, if $X$ is basic then $[X,U]\in\Gamma(\mathcal{V})$ for $U\in\Gamma(V)$, and taking account of \eqref{6} we get
\begin{eqnarray}
g(A_XU,Y)&=&g(\nabla_XU,Y)\nonumber\\
&=&g([X,U]+\nabla_UX,Y)\nonumber\\
&=&g(\nabla_UX,Y),\nonumber
\end{eqnarray}
which implies that $\nabla_UX=A_XU$ if $X$ is basic. Therefore, we have
\begin{eqnarray}
g(A_{PX}U,Y)&=&g(\nabla_U(PX),Y)\nonumber\\
&=&g((\nabla_UP)X+P\nabla_UX,Y)\nonumber\\
&=&g(PA_XU,Y).\nonumber
\end{eqnarray}

Thus we find $A_{PX}U=PA_XU$ and similarly we derive
\begin{equation}\label{37}
A^*_{PX}U=P^* A^*_XU,
\end{equation}
provided that $X$ is basic.


Making now use of \eqref{15} and  \eqref{37} in \eqref{33b} we deduce
\[
    g((\widehat{P}+\widehat{P}^*)A_XY,U)=g(Y,A^*_{P^*X}U)+g(A_{P^*X}Y,U)
\]
and taking account of \eqref{1} and \eqref{6}, we obtain
\begin{eqnarray}
    g((\widehat{P}+\widehat{P}^*)A_XY,U)&=&g(Y,\nabla^*_{P^*X}U)+g(\nabla_{P^*X}Y,U)\nonumber\\
    &=&(P^*X)g(Y,U)\nonumber\\
    &=&0.\nonumber
    \end{eqnarray}

Hence, we deduce that
\[
(\widehat{P}+\widehat{P}^*)A_XY=0
\]
and the conclusion is now clear taking account of \eqref{9}.
\end{proof}

\begin{cor}\label{C1}
If  $\pi:M\rightarrow M'$ is a para-K\"{a}hler-like statistical submersion such that $\widehat{P}=\widehat{P}^*$, then
$A_XY=A^*_XY=0$, for all $X,Y\in\Gamma(\mathcal{H})$.
\end{cor}
\begin{proof}
This assertion  is clear  from Theorem \ref{T3}.
\end{proof}

\begin{cor}\label{C2}
If $\pi:M\rightarrow M'$ is a para-K\"{a}hler-like statistical submersion such that $\widehat{P}=\widehat{P}^*$, then the horizontal distribution $\mathcal{H}$ is completely integrable.
\end{cor}
\begin{proof}
The conclusion follows immediately from Corollary \ref{C1} and \eqref{8}.
\end{proof}

\begin{rem}
We note that Corollary \ref{C2} generalizes Theorem 3.2 from \cite{GUND}. Indeed, if  $\pi:M\rightarrow M'$ is a para-K\"{a}hler submersion, then the condition $\widehat{P}=\widehat{P}^*$ is trivially satisfied and it follows that $\mathcal{H}$ is completely integrable.
\end{rem}

\begin{thm}
Let $\pi:M\rightarrow M'$ be a para-K\"{a}hler-like statistical submersion, such that the total space of the submersion is of type para-K\"{a}hler space form. Then:\\
i.  If ${\rm rank}(\widehat{P}+\widehat{P}^*)$ coincides with the dimension of the fibers, then the base space of the submersion is of type para-K\"{a}hler space form.\\
ii. If $\pi$ is a statistical submersion with isometric fibers, then each fiber is a totally geodesic submanifold of the total space $M$, of type para-K\"{a}hler space form.\\
iii. If $\pi$ is a statistical submersion with isometric fibers such that ${\rm rank}(\widehat{P}+\widehat{P}^*)$ coincides with the dimension of the fibers, then the base space and each fiber are flat. Moreover, the total space of the submersion is a locally product space of the base space and fiber.
\end{thm}
\begin{proof}
The conclusions follow from the analogues of the O'Neill equations for a statistical submersion \cite[Theorem 2.1]{TAK2} and taking account of Theorem \ref{T3}.
\end{proof}

\section{Examples of para-K\"{a}hler-like statistical manifolds and submersions}

\begin{ex}
According to the Remark \ref{R1}, any para-K\"{a}hler manifold is a para-K\"{a}hler-like statistical manifold. Notice that several examples of para-K\"{a}hler manifolds are collected in \cite{CFG}. See also \cite{CHEN2017}.
\end{ex}

\begin{ex}
We consider the canonical coordinates $(x_1,\ldots,x_n,y_1,\ldots,y_n)$ on the manifold $\mathbb{R}^{2n}$ equipped with the flat affine connection $\nabla$, and define a semi-Riemannian metric $g$ on $\mathbb{R}^{2n}$ by
\[
g=\sum_{i=1}^{n}\epsilon_i(k{\rm d}x_i^2-{\rm d}y_i^2),
\]
for an arbitrary non-zero real constant $k$, where $\epsilon_1,\ldots,\epsilon_n\in\{-1,+1\}$.

We also define an almost product structure $P$ on $\mathbb{R}^{2n}$ by
\[
P(\partial_{x_i})=\partial_{y_i},\ P(\partial_{y_i})=\partial_{x_i}.
\]

Then it follows immediately that the quadruplet $(\mathbb{R}^{2n},\nabla,P,g)$ defined above is a para-K\"{a}hler-like statistical manifold. We note that
the conjugate connection is also flat and the almost product structure $P^*$ is given by
\[
P^*(\partial_{x_i})=k\partial_{y_i},\ P^*(\partial_{y_i})=\frac{1}{k}\partial_{x_i}.
\]

Moreover, it is obvious that $(\mathbb{R}^{2n},\nabla,P,g)$ is a flat para-K\"{a}hler-like statistical manifold.
\end{ex}

\begin{ex}\label{EX1}
We consider the $2n$-dimensional semi-Riemannian manifold $(\mathbb{M}^{2n},g)$, where
\[\mathbb{M}^{2n}=\{(x_1,y_1,\ldots,x_n,y_n)\in\mathbb{R}^{2n}|y_1\neq 0,\ldots,y_n\neq 0\}\]
and the metric $g$ is given by
\[
g=\sum_{i=1}^{n}\frac{\epsilon_i}{y_i^2}(k{\rm d}x_i^2-l{\rm d}y_i^2),
\]
for arbitrary non-zero real constants $k$ and $l$,  where $\epsilon_1,\ldots,\epsilon_n\in\{-1,+1\}$.


Next we suppose that $k+l\neq0$ and define an affine connection $\nabla$ on manifold as follows
\[
\nabla_{\partial_{x_i}}\partial_{x_j}=\nabla_{\partial_{y_i}}\partial_{y_j}=-\frac{2k\delta_{ij}}{(k+l)y_i}\partial_{y_i},
\]
\[
\nabla_{\partial_{x_i}}\partial_{y_j}=\nabla_{\partial_{y_j}}\partial_{x_i}=-\frac{2k\delta_{ij}}{(k+l)y_i}\partial_{x_i}.
\]

We also define an almost product structure $P$  by
\[
P(\partial_{x_i})=\partial_{y_i},\ P(\partial_{y_i})=\partial_{x_i}.
\]

Now we can easily check that $(\mathbb{M}^{2n},\nabla,P,g)$ is a para-K\"{a}hler-like statistical manifold. In particular, we find
the almost product structure $P^*$ given by
\[
P^*(\partial_{x_i})=\frac{k}{l}\partial_{y_i},\ P^*(\partial_{y_i})=\frac{l}{k}\partial_{x_i}
\]
and the conjugate connection $\nabla^*$ defined by
\[
\nabla^*_{\partial_{x_i}}\partial_{x_j}=-\frac{2k^2\delta_{ij}}{l(k+l)y_i}\partial_{y_i},\ \nabla^*_{\partial_{y_i}}\partial_{y_j}=-\frac{2l\delta_{ij}}{(k+l)y_i}\partial_{y_i}
\]
\[
\nabla^*_{\partial_{x_i}}\partial_{y_j}=\nabla^*_{\partial_{y_j}}\partial_{x_i}=-\frac{2l\delta_{ij}}{(k+l)y_i}\partial_{x_i}.
\]

Moreover, we remark that this example proves that we can construct para-K\"{a}hler-like statistical manifolds of any signature, unlike the para-K\"{a}hler manifolds which are always of neutral signature.
\end{ex}

\begin{ex}\label{ex1b}
If $(M,P,g)$ is an almost para-Hermitian-like manifold endowed with a statistical structure $(\nabla,g)$, then we prove that the tangent bundle $TM$ is an almost para-Hermitian-like manifold that can be endowed with a statistical structure. We consider on the tangent bundle $TM$
the Sasaki metric $G$ defined by
\[
G(A,B)=g(KA,KB)+g(\pi_*A,\pi_*B),
\]
for all vector fields $A,B$ on $TM$, where $\pi$ is the natural projection of $TM$ onto $M$ and $K$ is
the connection map associated with the Levi-Civita connection of the metric $g$ (see \cite{D}).

We note that if $X\in\Gamma(TM)$, then there exists exactly one
vector field on $TM$, denoted by $X^h$ and called the \emph{horizontal lift}, resp. denoted $X^v$ and called the \emph{vertical
lift} of $X$, such that we have for all $U\in TM$:
\[
\pi_*X_U^h=X_{\pi(U)},\ \pi_*X_U^v=0_{\pi(U)},\ KX_U^h=0_{\pi(U)},\ KX_U^v=X_{\pi(U)}.
\]

It is known from \cite[Theorem 3]{II} that one can define a torsion free linear connection $\nabla'$ on $TM$ compatible to the Sasaki metric $G$. Hence $(TM,\nabla',G)$ is a statistical manifold. Using the almost product structure $P$ on $M$, we can define a tensor field $P'$ of type $(1,1)$ on $TM$ by
\begin{equation}\nonumber
       \left\{\begin{array}{lcr}
       P' X^h=(P X)^h\\
       P' X^v=(P X)^v
       \end{array}\right..
       \end{equation}

It is easy to see now that $P'$ is almost product structure on $TM$ and $(TM,P',G)$ is an almost para-Hermitian-like manifold. Moreover, it follows that $(TM,\nabla',P',G)$ is a para-K\"{a}hler-like statistical manifold if and only if $(M,\nabla,P,g)$
is a flat para-K\"{a}hler-like statistical manifold.
\end{ex}

\begin{ex}\label{EX3}
Let $M=\{p_{\xi}|\xi\in\Xi\}$ be an $n$-dimensional statistical model such that probability distributions $p_{\xi}=p(x;\xi)$ can be expressed in terms of functions $\{C,F_1,\ldots,F_n\}$ on $\chi$ and a function $\psi$ on $\Xi$ as
\begin{equation}
p(x;\xi)=\exp\left[C(x)+\sum_{i=1}^{n}\xi^iF_i(x)-\psi(\xi)\right].
\end{equation}

Then the statistical model $M$ is said to be an  \emph{exponential family} (see \cite{AMN}) and from the normalization condition $\int_\chi p(x;\xi)dx=1$ we derive
\[
\psi(\xi)=\log\int_{\chi}\exp\left[C(x)+\sum_{i=1}^{n}\xi^iF_i(x)\right]dx.
\]

Let $(M,\nabla^{(\alpha)},g)$ be a statistical manifold of the exponential family $M$. Then from Remark \ref{remi} it follows that the Fisher metric $g$ and the $\alpha$-connection are given respectively by
\[
g_{ij}=\partial_i\partial_j\psi
\]
and
\[
\nabla^{(\alpha)}_{\partial_i}\partial_j=\frac{1-\alpha}{2}\partial_sg_{ij}g^{st}\partial_t,
\]
where $g^{st}$ are the components of $g^{-1}$.

We define now an almost product structure $P^{(1)}$ on $(M,\nabla^{(1)},g)$ by components
\[
(P^{(1)})_{i}^{j}=a_{i}^{j},
\]
where $a_{i}^{j}$ are real constants satisfying $a_{i}^{k}a_{k}^{j}=\delta_{i}^{j}$. Then it follows immediately that $(M,\nabla^{(1)},P^{(1)},g)$ is a para-K\"{a}hler-like statistical manifold.
We can also define an almost product structure $P^{(-1)}$ on $(M,\nabla^{(-1)},g)$ by components
\[
(P^{(-1)})_{i}^{j}=-a_{s}^{k}g_{ki}g^{sj}
\]
and then we find that $(M,\nabla^{(-1)},P^{(-1)},g)$ is a para-K\"{a}hler-like statistical manifold.

Due to the fact that several standard statistical models are shown to belong to the exponential family, including the
well-known normal distribution, Poisson distribution, multinomial distribution, multivariate normal distribution, Dirichlet distribution and Von Mises-Fisher distribution (see \cite{AMN,TAK3}), it follows that the corresponding statistical manifolds $(M,\nabla^{(\alpha)},g)$ for all these distributions are para-K\"{a}hler-like statistical manifolds, provided that $\alpha=\pm1$.
\end{ex}

\begin{ex}
Let $(\mathbb{M}^{2n},\nabla,P,g)$ be the para-K\"{a}hler-like statistical manifold constructed in Example \ref{EX1} having signature $(r,2n-r)$. If we consider a similar para-K\"{a}hler-like statistical manifold $(\mathbb{M}^{2m},\nabla,P,g)$ of dimension $m$ and with signature $(s,2m-s)$, such that $m<n$ and $s\leq r$, then it follows easily that
the map $\pi:\mathbb{M}^{2n}\rightarrow\mathbb{M}^{2m}$ defined by
\[
\pi(x_1,y_1,\ldots,x_n,y_n)=(x_1,y_1,\ldots,x_m,y_m)
\]
is a para-K\"{a}hler-like statistical submersion with isometric fibers.
\end{ex}

\begin{ex}
Let $(M,P,g)$ be an almost para-Hermitian-like manifold endowed with a statistical structure $(\nabla,g)$ and
$(TM,P',G)$ be the almost para-Hermitian-like manifold equipped with the statistical structure $(\nabla',G)$ constructed in Example \ref{ex1b}. Then the canonical projection $\pi:TM\rightarrow M$ is a para-holomorphic map because
\[
\pi_*P' X^v=\pi_*(PX)^v=0=P\pi_*X^v,
\]
\[
\pi_*P' X^h=\pi_*(P X)^h=P X=P\pi_*X^h.
\]

Hence, indeed we have $\pi_*\circ P'=P\circ\pi_*$ and one can deduce now easily that $\pi$ is an almost para-Hermitian-like statistical submersion. Moreover, it follows that $\pi$ is a para-K\"{a}hler-like statistical submersion if and only if $(M,\nabla,P,g)$ is a flat para-K\"{a}hler-like statistical manifold.
\end{ex}

\section*{Acknowledgement}
The author would like to thank the anonymous reviewer for the thoughtful comments on the manuscript. This work was supported by a grant of Ministry of Research and Innovation, CNCS-UEFISCDI, project number PN-III-P4-ID-PCE-2016-0065, within PNCDI III.

Gabriel Eduard V\^{I}LCU$^{1,2}$ \\
      $^1$University of Bucharest, Faculty of Mathematics and Computer Science,\\
      Research Center in Geometry, Topology and Algebra,\\
      Str. Academiei, Nr. 14, Sector 1, Bucure\c sti 70109-ROMANIA\\
      e-mail: gevilcu@yahoo.com\\
      $^2$Petroleum-Gas University of Ploie\c sti,\\
      Department of Cybernetics, Economic Informatics, Finance and Accountancy,\\
      Bd. Bucure\c sti, Nr. 39, Ploie\c sti 100680-ROMANIA\\
      e-mail: gvilcu@upg-ploiesti.ro

\end{document}